\documentclass[11pt]{amsart}
\usepackage{graphicx}

\newtheorem{thm}{Theorem}[section]
\newtheorem{cor}[thm]{Corollary}
\newtheorem{lem}[thm]{Lemma}
\newtheorem{prop}[thm]{Proposition}
\theoremstyle{definition}
\newtheorem{defn}[thm]{Definition}

\theoremstyle{remark}
\newtheorem{rem}[thm]{Remark}
\numberwithin{equation}{section}

\begin{document}
\title{On maximal regularity and semivariation of $\alpha$-times resolvent families}
\author{Fu-Bo Li and Miao Li}
\address{Department of Mathematics, Sichuan University, Chengdu, Sichuan 610064, P.R. China}
\email{lifubo@scu.edu.cn; mli@scu.edu.cn}

\thanks{2000 {\it Mathematics Subject Classification}. Primary 45N05; Secondary 26A33, 34G10.\\
\mbox{     }\hspace{0.3cm}{\it Key words and phrases.}
$\alpha$-times resolvent family, maximal regularity, semivariation
\\
 \mbox{     }\hspace{0.3cm}The authors were
supported by the NSFC-RFBR Programme (Grant No. 108011120015).}

\begin{abstract}
Let $1< \alpha <2$ and $A$ be the generator of an $\alpha$-times
resolvent family $\{S_\alpha(t)\}_{t \ge 0}$ on a Banach space $X$.
It is shown that the fractional Cauchy problem ${\bf D}_t^\alpha
u(t) = Au(t)+f(t)$, $t \in [0,r]$; $u(0), u'(0) \in D(A)$ has
maximal regularity on $C([0,r];X)$ if and only if $S_\alpha(\cdot)$
is of bounded semivariation on $[0,r]$.
\end{abstract}

\maketitle
\section{Introduction}
Many initial and boundary value problems can be reduced to an
abstract Cauchy problem of the form
\begin{equation}\label{ACPf-first}
\begin{array}{ll}
u'(t) = Au(t) + f(t),  & t \in [0,r] \\
u(0)=x \in D(A) &
\end{array}
\end{equation}
where $A$ is the generator of a $C_0$-semigroup. One says that
(\ref{ACPf-first}) has maximal regularity on $C([0,r];X)$ if for
every $f \in C([0,r];X)$ there exists a unique $u \in C^1([0,r];X)$
satisfying (\ref{ACPf-first}). From the closed graph theorem it
follows easily that if there is maximal regularity on $C([0,r];X)$,
then there exists a constant $C>0$ such that
$$
\|u'\|_{C([0,r];X)} + \|Au\|_{C([0,r];X)} \le \|f\|_{C([0,r];X)}.
$$
Travis \cite{Tr} proved that the maximal regularity is equivalent to
 the $C_0$-semigroup generated by $A$ being of bounded
semivariation on $[0,r]$.

Chyan, Shaw and Piskarev \cite{CSP} gave similar results for second
order Cauchy problems. More precisely, they showed that the second
order Cauchy problem
\begin{equation}\label{ACPf-second}
\begin{array}{ll}
u''(t) = Au(t) + f(t),  & t \in [0,r] \\
u(0)=x,\, u'(0)=y, & x,y \in D(A)
\end{array}
\end{equation}
has maximal regularity on $[0,r]$ if and only if the cosine opeator
function generated by $A$ is of bounded semivariation on $[0,r]$.

 In this paper we will consider the maximal regularity for
fractional Cauchy problem
\begin{equation}\label{ACPf-alpha}
\begin{array}{ll}
{\bf D}_t^{\alpha}u(t) = Au(t) + f(t),  & t \in [0,r] \\
u(0)=x, \, u'(0)=y, & x,y \in D(A)
\end{array}
\end{equation}
where $\alpha \in (1,2)$, $A$ is the generator of an $\alpha$-times
resolvent family (see Definition \ref{alpha-resolvent} below) and
${\bf D}_t^\alpha u$ is understood in the Caputo sense.
 We show that (\ref{ACPf-alpha}) has maximal
regularity on $C([0,r];X)$ if and only if the corresponding
$\alpha$-times resolvent family is of bounded semivariation on
$[0,r]$.

\section{Preliminaries}

 Let $1<\alpha <2$, $g_0(t):=\delta(t)$ and
$g_\beta(t):=\frac{t^{\beta-1}}{\Gamma(\beta)}(\beta>0)$ for $t>0$.
Recall the Caputo fractional derivative of order $\alpha>0$
\begin{eqnarray*}
{\bf D}_t^\alpha f(t):=\int_0^t
g_{2-\alpha}(t-s)\frac{d^2}{ds^2}f(s)ds, \quad t \in [0,r]
\end{eqnarray*}
for $f \in C^2([0,r];X)$. The condition that $f \in C^2([0,r];X)$
can be relaxed to $f \in C^1([0,r];X)$ and $g_{2-\alpha}* (f- f(0) -
f'(0)g_2) \in C^2([0,r];X)$, for details and further properties see
\cite{Baj} and references therein. And in the above we denote by
\begin{eqnarray*}
(g_\beta*f)(t)=\int_0^tg_\beta(t-s)f(s)ds
\end{eqnarray*}
the convolution of $g_\beta$ with $f$. Note that $g_\alpha* g_\beta
= g_{\alpha +\beta}$.

Consider a closed linear operator $A$ densely defined in a Banach
space $X$ and the fractional evolution equation (\ref{ACPf-alpha}).

\begin{defn}\label{defn-solution}
 A function $u \in C([0,r];X)$ is called a {\it strong
solution} of (\ref{ACPf-alpha}) if
$$
u \in C([0,r]; D(A))\cap C^1 ([0,r];X), \quad g_{2-\alpha}*(u(t) -
x-ty) \in C^2([0,r];X)
$$
and (\ref{ACPf-alpha}) holds on $[0,r]$.
$u\in C([0,r]; X)$ is called a {\it mild solution} of
(\ref{ACPf-alpha}) if $g_\alpha * u \in D(A)$ and
$$
u(t) - x - ty = A(g_\alpha *u)(t) + (g_\alpha*f)(t)
$$
for $t \in [0,r]$.
\end{defn}
\begin{defn}\label{alpha-resolvent}
Assume that $A$ is a closed, densely defined linear operator on $X$.
A family $\{S_\alpha(t)\}_{t\geq0}\subset B(X)$ is called an {\it
$\alpha$-times resolvent family} generated by $A$ if the following
conditions are satisfied:

(a) $S_\alpha(\cdot)$ is strongly continuous on $\mathbb{R}_+$ and
$S_\alpha(0)=I$;

(b) $S_\alpha(t)D(A)\subset D(A)$ and $AS_\alpha(t)x=S_\alpha(t)Ax$
for all $x\in D(A),t\geq0$;

(c) For all $x\in D(A)$ and $t\geq0$, $S_\alpha(t)x=x+ (g_\alpha
* S_\alpha)(t)Ax$.
\end{defn}

\begin{rem}\label{g*S}
Since $A$ is closed and densely defined, it is easy to show that for
all $x \in X$, $(g_\alpha * S_\alpha)(t)x \in D(A)$ and $ A(g_\alpha
* S_\alpha)(t) x = S_\alpha x  - x. $
\end{rem}

The alpha-times resolvent families are closely related to the
solutions of (\ref{ACPf-alpha}). It was shown in \cite{Baj} that if
$A$ generates an $\alpha$-times resolvent family $S_\alpha(\cdot)$,
then (\ref{ACPf-alpha}) has a unique strong solution given by
$S_\alpha(t)x + \int_0^t S_\alpha(s)yds$.\\

Next we recall the definition of functions of bounded semivariation
(see e.g. \cite{Ho}). Given a closed interval $[a,b]$ of the real
line, a subdivision of $[a,b]$ is a finite sequence $d: a= d_0 < d_1
< \cdots < d_n=b$. Let $D[a,b]$ denote the set of all subdivisions
of $[a,b]$.

\begin{defn} For $G: [a,b] \to B(X)$ and $d \in D[a,b]$, define
$$
SV_d[G]= \sup\{\|\sum_{n=1}^n [G(d_i)-G(d_{i-1})]x_i\|: x_i \in X,
\|x_i\| \le 1\}
$$
and $SV[G] = \sup \{SV_d[G]: d \in D[a,b]\}$. We say $G$ is of
bounded sevivariation if $SV[G] < \infty$.
\end{defn}

\section{Main results}

We begin with some properties on $\alpha$-times resolvent families
which will be needed in the sequel.

\begin{prop}\label{alpha-prop}
Let $1< \alpha < 2$ and $\{S_\alpha(t)\}_{t \ge 0}$ be the
$\alpha$-times resolvent family with generator $A$. Define
$$
P_\alpha (t)x= (g_{\alpha -1} * S_\alpha)(t)x= \int_0^t
g_{\alpha-1}(t-s)S_\alpha(s)x ds,\quad x \in X,
$$
then the following statements are true.

 (a) For every $x \in X$, $\int_0^t P_\alpha(s)x ds \in D(A)$ and
$$
A \int_0^t P_\alpha(s)x ds = S_\alpha(t)x - x;
$$

(b) For every $x \in X$, $0\le a, b \le t$, $\int_a^b s
P_\alpha(t-s) xdx \in D(A)$ and
$$
A\int_a^b sP_\alpha(t-s)xds = aS_\alpha(t-a)x - bS_\alpha(t-b)x +
\int_a^b S_\alpha(t-s)xds;
$$

(c) For every $x \in X$, $\int_0^t g_\alpha(t-s)sP_\alpha(s)x ds \in
D(A)$ and
$$
A\Big(\int_0^t g_\alpha(t-s)sP_\alpha(s)x ds \Big)= -\alpha
(g_\alpha*S_\alpha)(t)x + t P_\alpha(t)x;
$$

(d) If $f \in C([0,r];X)$, then $g_\alpha* S_\alpha *f \in D(A)$ and
\begin{equation}\label{convolution}
A(g_\alpha * S_\alpha *f) = (S_\alpha - 1)* f.
\end{equation}
\end{prop}
\begin{proof}
(a) follows from the fact that $\int_0^t P_\alpha(s)x ds = (g_1
* g_{\alpha -1}*S_\alpha)(t)x= (g_\alpha * S_\alpha)(t)x \in D(A)$
and $A (g_\alpha*S_\alpha)(t)x = S_\alpha(t)x -x$ by Remark
\ref{g*S}.

(b) By integration by parts we have
\begin{eqnarray*}
\int_a^b sP_\alpha(t-s)xds &=& \int_a^b s d_s [\int_0^s
P_\alpha(t-\tau)x d\tau] \\
&=&  \int_a^b sd_s [(g_\alpha*S_\alpha)(t-s)x]\\
&=& - s(g_\alpha*S_\alpha)(t-s)x\Big|_a^b + \int_a^b
(g_\alpha*S_\alpha)(t-s)x ds\\
&=& a (g_\alpha*S_\alpha)(t-a)x - b (g_\alpha*S_\alpha)(t-b)x+
\int_a^b (g_\alpha*S_\alpha)(t-s)x ds,
\end{eqnarray*}
since $ (g_\alpha*S_\alpha)(t)x ds \in D(A)$  by Remark \ref{g*S},
operating $A$ on both sides of the above identity gives (b).

 (c) follows from the fact that
\begin{eqnarray*}
&&\int_0^t g_\alpha(t-s) sP_\alpha(s)x ds \\
&=& \int_0^t g_\alpha(t-s)(s-t)P_\alpha(s)x ds + t\int_0^t
g_\alpha(t-s)P_\alpha(s)x ds \\
&=& - \alpha\int_0^t g_{\alpha+1}(t-s)P_\alpha(s)x ds +
t(g_\alpha*P_\alpha)(t)x\\
&=& -\alpha(g_{\alpha+1}*P_\alpha)(t)x +t(g_\alpha*P_\alpha)(t)x \\
&=&-\alpha(g_{\alpha+1}*g_{\alpha-1}*S_\alpha)(t)x
+t(g_\alpha*g_{\alpha-1}* S_\alpha)(t)x \\
&=&-\alpha(g_{\alpha}*g_{\alpha}*S_\alpha)(t)x
+t(g_{\alpha-1}*g_{\alpha}* S_\alpha)(t)x
\end{eqnarray*}
belongs to $D(A)$ and
\begin{eqnarray*}
A(\int_0^t g_\alpha(t-s) sP_\alpha(s)x ds)&=&-\alpha
(g_{\alpha}*A(g_{\alpha}*S_\alpha))(t)x
+t(g_{\alpha-1}*A(g_{\alpha}*
S_\alpha))(t)x \\
&=& -\alpha (g_\alpha*(S_\alpha - 1))(t)x + t (g_{\alpha-1}
*(S_\alpha
- 1))(t)x\\
&=& - \alpha (g_\alpha*S_\alpha)(t)x + \alpha g_{\alpha+1}(t)x + t
(g_{\alpha-1}*S_\alpha)(t) - tg_{\alpha}(t)x \\
&=&- \alpha (g_\alpha*S_\alpha)(t)x+ t P_\alpha(t)x.
\end{eqnarray*}

 (d) (\ref{convolution}) is true for step
functions, and then for continuous functions by the closedness of
$A$.
\end{proof}

The following two lemmas can be proved similarly as that in
\cite{CSP, Tr}.

\begin{lem}\label{P*finD(A)}
If $f \in C([0,r];X)$ and the $\alpha$-times resolvent family
$S_\alpha(t)$ is of bounded semivariation on $[0,r]$, then
$(P_\alpha
* f)(t) \in D(A)$ and
$$
A(P_\alpha * f)(t) = -\int_0^t  d_s [S_\alpha(t-s)]f(s).
$$
\end{lem}

\begin{lem}\label{APcont}
If $f \in C([0,r];X)$ and the $\alpha$-times resolvent family
$S_\alpha(t)$ is of bounded semivariation on $[0,r]$, then $\int_0^t
d_s[S_\alpha(t-s)]f(s)$ is continuous in $t$ on $[0,r]$.
\end{lem}

We next turn to the solution of
\begin{equation}\label{ACPf0}
\begin{split}
&{\bf D}_t^\alpha u(t) = Au(t)+f(t), \quad t\in [0,r], \\
&u(0)=0,\, u'(0)=0,
\end{split}
\end{equation}
where $A$ is the generator of an $\alpha$-times resolvent family. If
$v(t)$ is a mild solution of (\ref{ACPf0}), then by Definition
\ref{defn-solution} $(g_\alpha
* v)(t) \in D(A)$ and $v(t) = A(g_\alpha
*v)(t) + (g_\alpha
*f)(t)$. It then follows from the properties of $\alpha$-times
resolvent family that
$$
1*v = (S_\alpha - A(g_\alpha* S_\alpha))*v= S_\alpha *v - S_\alpha *
A(g_\alpha*v)= S_\alpha*(v - A(g_\alpha*v))= S_\alpha* g_\alpha *f,
$$
which implies that $g_\alpha * S_\alpha *f$ is differentiable and
$$
v(t) = \frac{d}{dt} (g_\alpha* S_\alpha *f)(t)=
(g_{\alpha-1}*S_\alpha*f)(t)= (P_\alpha *f)(t).
$$
Therefore, the mild solution of (\ref{ACPf-alpha}) is given by
\begin{equation}\label{strong-solution}
u(t) = S_\alpha(t)x +\int_0^t S_\alpha(s)y ds+ (P_\alpha * f)(t).
\end{equation}

\begin{prop} \label{sPSf}Let $A$ be the generator of an $\alpha$-times resolvent
family $S_\alpha(\cdot)$, and let
 $f \in C([0,r];X)$ and
$x,y \in D(A)$. Then the following statements are equivalent:

(a) (\ref{ACPf-alpha}) has a strong solution;

(b) $(S_\alpha * f)(\cdot) \in C^1([0,r];X)$;

(c) $(P_\alpha * f)(t) \in D(A)$ for $0 \le t \le r$ and $A(P_\alpha
* f)(t)$ is continuous in $t$ on $[0,r]$.

\end{prop}

\begin{proof}
(a) If $u(t)$ is a strong solution of (\ref{ACPf-alpha}), then $u$
is given by (\ref{strong-solution}) since every strong solution is a
mild solution. Therefore, by the definition of strong solutions,
$g_{2-\alpha}*P_\alpha*f= g_1*S_\alpha*f \in C^2([0,r];X)$; it then
follows that $S_\alpha * f \in C^1([0,r];X)$, this is (b).

$(b) \Rightarrow (c)$. Suppose that $S_\alpha * f \in C^1([0,r];X)$.
Since $g_1 * P_\alpha *f = g_\alpha * S_\alpha *f$, by Proposition
\ref{alpha-prop}(d), $g_1 * P_\alpha * f \in D(A)$ and
\begin{equation}\label{S*f}
A(g_1*P_\alpha *f) = A(g_\alpha * S_\alpha *f)= (S_\alpha - 1)*f.
\end{equation}
Since $A$ is closed and $S_\alpha * f\in C^1([0,r];X)$, we have
$P_\alpha * f \in D(A)$ and $A (P_\alpha * f) = (S_\alpha*f)' - f$
is continuous.

$(c) \Rightarrow (a)$. By (\ref{S*f}), $g_1* A(P_\alpha
* f) =A(g_1*P_\alpha
*f) = (S_\alpha - 1)*f$, therefore $S_\alpha * f$ is differentiable and thus
$g_{2-\alpha}*P_\alpha*f=  g_1*S_\alpha*f$ is in $C^2([0,r];X)$. It
is easy to check that $u(t)$ defined by (\ref{strong-solution}) is a
strong solution of (\ref{ACPf-alpha}).
\end{proof}

Now we are in the position to give the main result of this paper.
The proof is similar to that of Proposition 3.1 in \cite{Tr} or
Theorem 4.2 in \cite{CSP}, we write it out for completeness.

\begin{thm}
Suppose that $A$ generates an $\alpha$-times resolvent family
$\{S_\alpha(t)\}_{t \ge 0}$. Then the function
(\ref{strong-solution}) is a strong solution of the Cauchy problem
(\ref{ACPf-alpha}) for every pair $x, y \in D(A)$ and continuous
function $f$ if and only if $S_\alpha(\cdot)$ is of bounded
semivariation on $[0,r]$.
\end{thm}
\begin{proof}
The sufficiency follows from Lemmas \ref{P*finD(A)} and
\ref{APcont}.

Conversely, suppose that for $x,y \in D(A)$ and continuous function
$f$, $u(t)$ given by (\ref{strong-solution}) is a strong solution
for (\ref{ACPf-alpha}). Define the bounded linear operator $L:
C([0,r];X) \to X$ by $L(f) = (P_\alpha*f)(r)$. By Proposition
\ref{sPSf} (c) $Lf \in D(A)$, it thus follows from the closedness of
$A$ that $AL: C([0,r];X) \to X$ is bounded.

Let $\{d_i\}_{i=0}^n$ be a subdivision of $[0,r]$ and $\epsilon>0$
such that $\epsilon < \min_{1\le i \le n}\{|d_{i}-d_{i-1}|\}$. For
$x_i \in X$ with $\|x_i\| \le 1$ ($i=1, 2, \cdots,  n+1$), define
$f_{d, \epsilon} \in C([0,r];X)$ by
$$
f_{d,\epsilon}(\tau) = \left\{\begin{array}{ll} x_i, & d_{i-1} \le
\tau \le d_i - \epsilon \\ x_{i+1} + \frac{\tau -
d_i}{\epsilon}(x_{i+1} - x_i), & d_i - \epsilon \le \tau \le
d_i\end{array}\right.,
$$
then $\|f_{d,\epsilon}\|_{C([0,r];X)} \le 1$. By Proposition
\ref{alpha-prop},
\begin{eqnarray*}
AL(f_{d,\epsilon}) &=& A \int_0^r P_\alpha(r-s)
f_{d,\epsilon}(s) ds\\
&=& \sum_{i=1}^n \Big[ A \int_{d_{i-1}}^{d_i -\epsilon}
P_\alpha(r-s)x_i ds \\
&& + A \int_{d_i -\epsilon}^{d_i} P_\alpha(r-s) x_{i+1} ds + A
\int_{d_i-\epsilon}^{d_i}
\frac{s-d_i}{\epsilon}P_\alpha(r-s)(x_{i+1}-x_i)dx\Big]\\
\end{eqnarray*}
\begin{eqnarray*}
&=& \sum_{i=1}^n \Big\{[S_\alpha(r-d_{i-1})x_i -
S_\alpha(r-d_{i}+\epsilon)x_i]\\
&& +[S_\alpha(r-d_{i}+\epsilon)x_{i+1} - S_\alpha(r-d_{i})x_{i+1}]
\\
&& - \frac{d}{\epsilon} [S_\alpha(r-d_{i}+\epsilon)(x_{i+1}-x_i) -
S_\alpha(r-d_{i})(x_{i+1}-x_i)] \\
&& + \frac{1}{\epsilon}[(d_i -
\epsilon)S_\alpha(r-d_{i}+\epsilon)(x_{i+1}-x_i) - d_i
S_\alpha(r-d_{i})(x_{i+1}-x_i)]\\
&& +\frac{1}{\epsilon} \int_{d_i - \epsilon}^{d_i} S_\alpha(r-s)
(x_{i+1} - x_i)ds\Big\} \\
&=& \sum_{i=1}^n \Big\{[S_\alpha(r-d_{i-1})x_i -
S_\alpha(r-d_{i})x_{i+1}]\\
&& +\frac{1}{\epsilon} \int_{d_i - \epsilon}^{d_i} S_\alpha(r-s)
(x_{i+1} - x_i)ds\Big\} \\
&=& \sum_{i=1}^n \Big\{[S_\alpha(r-d_{i-1})- S_\alpha(r-d_i)]x_i -
S_\alpha(r-d_{i})(x_{i+1}-x_i)\\
&& +\frac{1}{\epsilon} \int_{d_i - \epsilon}^{d_i} S_\alpha(r-s)
(x_{i+1} - x_i)ds\Big\},
\end{eqnarray*}
it then follows that
\begin{eqnarray*}
&&\Big\|\sum_{i=1}^n  [S_\alpha(r-d_{i-1})- S_\alpha(r-d_i)]x_i\Big\|\\
& \le& \|AL(f_{d,\epsilon})\|+ \sum_{i=1}^n
\Big\|S_\alpha(r-d_{i})(x_{i+1}-x_i)
 -\frac{1}{\epsilon} \int_{d_i - \epsilon}^{d_i} S_\alpha(r-s)
(x_{i+1} - x_i)ds\Big\|.
\end{eqnarray*}
By letting $\epsilon \to 0$, we obtain that $S_\alpha$ is of bounded
semivariation on $[0,r]$.
\end{proof}

\begin{cor}
Suppose that $\{S_\alpha(t)\}_{t \ge 0}$ is an $\alpha$-times
resolvent family with generator $A$ and $S_\alpha(\cdot)$ is of
bounded semivariation on $[0,r]$ for some $r>0$. Then
$R(P_\alpha(t))\subset D(A)$ for $t \in [0,r]$ and $\|t
AP_\alpha(t)\|$ is bounded on $[0,r]$.
\end{cor}
\begin{proof}
 For $x \in X$, consider $f(t) =\alpha
S_\alpha(t)x$. By Proposition \ref{alpha-prop}(c), $tP_\alpha(t)x$
is a mild solution of (\ref{ACPf0}). Moreover, it follows from
Proposition \ref{sPSf} that $P_\alpha *f$ is a strong solution of
(\ref{ACPf0}). Since a strong solution must be a mild solution, we
have $(P_\alpha *f)(t) = tP_\alpha(t)x$. Thus our claim follows from
Proposition \ref{sPSf}.
\end{proof}

\begin{rem}
Let $\alpha=1$. If $A$ generates a $C_0$-semigroup $T(\cdot)$, then
the condition that $tAT(t)$ is bounded on $[0,r]$ implies that
$T(\cdot)$ is analytic (see \cite{Pa}). When $\alpha=2$ and $A$
generates a cosine function $C(\cdot)$, then the condition that
$tAC(t)$ is bounded on $[0,r]$ implies that $A$ is bounded
(\cite{CSP}). However, since there is no semigroup properties for
$\alpha$-times resolvent family, it is not clear that one can get
the analyticity of $S_\alpha(\cdot)$ from the local boundedness of
$tAP_\alpha(t)$.
\end{rem}

\end{document}